\documentclass[11pt,english]{amsart}
\usepackage[T1]{fontenc}
\usepackage[latin9]{inputenc}
\usepackage{amstext}
\usepackage{amsthm}
\usepackage{amssymb}
\usepackage{esint}

\makeatletter
\usepackage{amsmath,amsfonts,amssymb,amsthm,epsfig}

\usepackage{color}
\usepackage[final,allcolors=blue,colorlinks=true]{hyperref}

\def\R{\mathbb R}
\def\N{\mathbb N}

\def\al{\alpha}
\def\be{\beta}
\def\ga{\gamma}
\def\de{\delta}
\def\ep{\epsilon}

\def\ta{\theta}

\def\om{\omega}
\def\na{\nabla}
\def\Ga{\Gamma}  
\def\Om{\Omega}  
\def\De{\Delta}      

\def\wq{\infty}
\def\pa{\partial}

\newcommand{\D}{{\rm d}}

\newcommand{\medint}{-\kern -,375cm\int}         
\newcommand{\medintinrigo}{-\kern -,315cm\int}

\numberwithin{equation}{section}
\newtheorem{theorem}{Theorem}[section]
\newtheorem*{theorem*}{Theorem}  

\newtheorem*{conclusion*}{Conclusin}

\newtheorem*{corollary*}{Corollary}

\newtheorem{lemma}[theorem]{Lemma}
\newtheorem*{lemma*}{Lemma}

\newtheorem*{notation*}{Notation}
\newtheorem{problem}[theorem]{Problem}
\newtheorem{proposition}[theorem]{Proposition}
\newtheorem*{proposition*}{Proposition}
\newtheorem{remark}[theorem]{Remark}
\newtheorem*{remark*}{Remark}

\newtheorem*{example*}{Example}                
\theoremstyle{definition}

\def\XXint#1#2#3{{\setbox0=\hbox{$#1{#2#3}{\int}$}
		\vcenter{\hbox{$#2#3$}}\kern-.5\wd0}}

\makeatother

\usepackage{babel}
\begin{document}
\title[Regularity of elliptic systems]{Regularity of weak solutions to higher order elliptic systems in critical dimensions}

\author[C.-Y. Guo and C.-L. Xiang]{Chang-Yu Guo and Chang-Lin Xiang$^\ast$}

\address[Chang-Yu Guo]{Institute of Mathematics, \'Ecole Polytechnique F\'ed\'erale de Lausanne (EPFL), Station 8,  CH-1015 Lausanne, Switzerland and Research Center for Mathematics and Interdisciplinary Sciences, Shandong University,  Qingdao, P. R. China}
\email{changyu.guo@sdu.edu.cn}

\address[Chang-Lin Xiang]{School of Information and Mathematics, Yangtze University, Jingzhou 434023, P.R. China}
\email{changlin.xiang@yangtzeu.edu.cn}
\thanks{$^\ast$ Corresponding author}
\thanks{C.-Y. Guo was supported by Swiss National Science Foundation Grant 175985 and the Qilu funding of Shandong University (No. 62550089963197). The corresponding author C.-L. Xiang is financially supported by the National Natural Science Foundation of China (No. 11701045) and  the Yangtze Youth Fund (No. 2016cqn56).}

\begin{abstract}
In this paper, we develop an elementary and unified treatment, in the spirit of Rivi\`ere and Struwe (Comm. Pure. Appl. Math. 2008), to explore regularity of weak solutions of higher order geometric elliptic systems in critical dimensions  without using conservation law. As a result,  we obtain an interior H\"older continuity for solutions of the higher order elliptic system of de Longueville and Gastel \cite{deLongueville-Gastel-2019} in critical dimensions
$$\Delta^{k}u=\sum_{i=0}^{k-1}\Delta^{i}\left\langle V_{i},du\right\rangle +\sum_{i=0}^{k-2}\Delta^{i}\delta\left(w_{i}du\right) \quad \text{in } B^{2k},$$
under critical regularity assumptions on the coefficient functions. This verifies an expectation of  Rivi\`ere   \cite[Page 108]{Riviere-2011}, and   provides an affirmative answer to an open question of Struwe \cite{Struwe-2008} in dimension four when $k=2$. The H\"older continuity is also an improvement of the continuity result of Lamm and  Rivi\`ere \cite{Lamm-Riviere-2008} and de Longueville and Gastel \cite{deLongueville-Gastel-2019}.
\end{abstract}

\maketitle

{\small
\smallskip
\newline
\tableofcontents}
\bigskip

\section{Introduction}\label{sec:introduction}
In the calculus of variations, functionals
$$u\mapsto \int F(x,u(x), Du(x))\D x$$
of quadratic growth are one of the most attractive topics, since many interesting geometric problems are concerned. By the fundamental work of Morrey \cite{Morrey-1948}, minimizers of such functionals in $W^{1,2}(B^2, \R^m)$ are locally H\"older continuous. However, when the domain has higher dimensions (greater than or equal to three), there exist discontinuous minimizers, much less to say about general critical points. Due to the geometric origin of many variational problems, a reasonable conjecture formulated by Hildbrandt \cite{Hildbrandt-1980} in the late 1970s is that \emph{ critical points of coercive conformally invariant Lagrangian with quadratic growth are regular.}  In his  pioneer work \cite{Helein-2002}, H\'elein   confirmed  this conjecture  in the case of weakly  harmonic mappings: every weakly harmonic mappings from the two dimensional disk $B^2\subset \R^2$ into any closed manifold are smooth via the nowadays well-known moving frame method.

In the recent  remarkable work \cite{Riviere-2007}, Rivi\`ere completely solved  this conjecture. In fact, his work has far more applications beyond conformally invariant problems; see \cite{Riviere-2011,Riviere-2012} for a comprehensive overview.  In  \cite{Riviere-2007},  he proposed the general second order linear elliptic  system \begin{equation}\label{eq:Riviere 2007}
	-\Delta u=\Omega\cdot \nabla u \qquad \text{in }B^2
\end{equation}
where $u\in W^{1,2}(B^2, \R^m)$ and $\Omega=(\Omega_{ij})\in L^2(B^2,so_m\otimes \Lambda^1\R^2)$. As was verified in \cite{Riviere-2007}, \eqref{eq:Riviere 2007} includes
the Euler-Lagrange equations of  critical points of all second order conformally invariant variational functionals which  act on mappings $u\in W^{1,2}(B^2,N)$ from  $B^2\subset \R^2$ into a closed Riemannian manifold $N\subset \R^m$. In particular, \eqref{eq:Riviere 2007} includes the equations of weakly harmonic mappings from  $B^2$ into $N$ and prescribed mean curvature equations.
 Note that the regularity assumption of $\Om$ makes the system critical in the sense that $\Om\cdot \na u\in L^1(B^2)$ which prevents  a direct application of the standard elliptic $L^p$-regularity theory.  To deduce continuity of weak solutions of \eqref{eq:Riviere 2007},  Rivi\`ere \cite{Riviere-2007} proved that each solution of  \eqref{eq:Riviere 2007} satisfies  certain conservation law,
which permits  to prove the continuity of weak solutions of \eqref{eq:Riviere 2007}, among many other results. As an application,  he recovered the regularity result of  H\'elein  \cite{Helein-2002}. 

 However, the conservation law does not hold in general in high dimensions (with respect to the domain side).  Consider the  same system in higher dimensions
 \begin{equation}\label{eq:Riviere 2007-2}
	-\Delta u=\Omega\cdot \nabla u \qquad \text{in }B^n
\end{equation}
where $n\ge 3$,  $u\in W^{1,2}(B^n,N)$ and $\Omega=(\Omega_{ij})\in L^2(B^n,so_m\otimes \Lambda^1\R^n)$.  Rivi\`ere and Struwe \cite{Riviere-Struve-2008} found that the conservation law of the above mentioned does not hold anymore. 
 To overcome this difficulty, they developed a new approach for the regularity based on the  work  \cite{Riviere-2007}. More precisely, they first adapted  a Gauge transform initially introduced by Uhlenbeck in her famous work \cite{Uhlenbeck-1982} (and further developed  in \cite{Riviere-2007}) to rewrite \eqref{eq:Riviere 2007-2} into an equivalent form, and then  derived  a partial regularity of weak solutions of \eqref{eq:Riviere 2007-2} from the Gauge-equivalent system directly,  avoiding the use of conservation law. This alternative approach also allowed them to reproduce the corresponding regularity  result of  Rivi\`ere \cite{Riviere-2007}.
As a byproduct of  \cite{Riviere-Struve-2008}, they reproduced partial H\"older continuity of stationary harmonic mappings from $B^n$ into general closed Riemannian manifolds; which was  proved earlier in \cite{Evans-1991,Bethuel-1993}.

Taking into account of conformal invariance in $\R^4$, it is natural to consider fourth order elliptic systems similar as \eqref{eq:Riviere 2007}.
 In their interesting work \cite{Lamm-Riviere-2008}, Lamm and Rivi\`ere considered the following fourth order elliptic system
\begin{equation}\label{eq:Lamm-Riviere 2008}
\De^{2}u=\De(V\cdot\na u)+{\rm div}(w\na u)+W\cdot\na u \quad  \text{in }B^4,
\end{equation}
where $V\in W^{1,2}(B^4,M_m\otimes \Lambda^1\R^{4})$, $w\in L^{2}(B^4,M_m)$,
and $W\in W^{-1,2}(B^4,M_m\otimes \Lambda^1\R^{4})$ is of the form
\[
W=\na\om+F,
\]
with $\om\in L^{2}(B^4,so_m)$ and $F\in L^{\frac{4}{3},1}(B^4,M_m\otimes \Lambda^1\R^{4})$. System \eqref{eq:Lamm-Riviere 2008} models many interesting fourth order conformally invariant geometric objects  in four dimensions such as extrinsic
and intrinsic biharmonic mappings from $B^4$ into closed Riemannian manifolds. As being critical points  of conformally invariant energy functionals in dimension  four, biharmonic mappings have attracted extensive attentions in recent years. Chang, Wang and Yang \cite{Chang-W-Y-1999} introduced and studied the regularity theory for extrinsic biharmonic mappings from $B^n$ into
Euclidean spheres and shortly after that, Wang developed the regularity theory of biharmonic mappings into general closed Riemannian manifolds in a series of pioneer works \cite{Wang-2004-CV,Wang-2004-MZ,Wang-2004-CPAM}; see also  \cite{Ku-2008,Strzelecki-2003} for more related works on biharmonic mappings.

The approach of Lamm and Rivi\`ere \cite{Lamm-Riviere-2008} follows closely the line of Rivi\`ere \cite{Riviere-2007}: they first established a conservation law for \eqref{eq:Lamm-Riviere 2008} similar to that of \eqref{eq:Riviere 2007},  and then they applied standard  potential theory to derive continuity of weak solutions to \eqref{eq:Lamm-Riviere 2008}. As an immediate consequence, they reproduced the continuity of weakly  biharmonic mappings from $B^4$ into a closed Riemannian manifold, which was initially proved by Wang \cite{Wang-2004-MZ} based on the moving frame method of  H\'elein \cite{Helein-2002}.

In view of the work of Rivi\`ere and Struwe \cite{Riviere-Struve-2008}, it is natural to approach the regularity of solutions  for   \eqref{eq:Lamm-Riviere 2008} without the use of  conservation law. An initial attempt towards this was made by  Struwe
 \cite{Struwe-2008}:
he first wrote the  system of bi-harmonic mappings in the form
\begin{equation}\label{eq: Struwe's equation}
  \De^{2}u=\De(D\cdot\na u)+{\rm div}(E\cdot \na u)+F\cdot\na u \qquad \text{in }B^n,
\end{equation}
where $n\ge 4$, $
F=\De\Om+G
$, $D, E, G$ belong to some function spaces and $\Om$ is an $so_m$-valued function with entries in $\Lambda^1\R^n$.
 Then he developed a  Gauge transform 
 to reformulate \eqref{eq: Struwe's equation} into a gauge-equivalent form. To deduce regularity of solutions to  \eqref{eq: Struwe's equation},  Struwe \cite{Struwe-2008} regarded the gauge-equivalent system together with the system for gauge transform as a coupled system (for both  $u$ and the Gauge $P$), from which he obtained  some decay estimates  for both $u$ and $P$. Then (local) partial  H\"older continuity of $u$ follows from  Morrey's Dirichlet growth theorem. However, in this approach, Struwe had to assume in addition that $D,E,G, \Om$ satisfy certain growth conditions in terms of  $u$ and its derivatives. Thus, this approach does not directly apply for the general elliptic system \eqref{eq: Struwe's equation}.  In \cite[Page 250]{Struwe-2008}, he proposed:

\emph{"It would be interesting to see if our method can be extended to general linear systems of fourth order that exhibit a structure similar to the one of equation \eqref{eq: Struwe's equation}, as is the case for second
order systems \eqref{eq:Riviere 2007-2} or in the conformal case $n=4$ considered in \cite{Lamm-Riviere-2008}."}

In other words, the above problem can be reformulated as follows:
\begin{problem}\label{prob:Struwe} 	Is it possible to extend the method of \cite{Riviere-Struve-2008} to the general fourth order linear system \eqref{eq: Struwe's equation} or to the ``conformal" case $n=4$ for the system \eqref{eq:Lamm-Riviere 2008}? \end{problem}

Conformal invariant problems also appear naturally in all even dimensions, which lead to higher order elliptic systems. One particular interest in this respect comes from the following  open problem proposed in Lamm and Rivi\`ere \cite[Remark 1.4]{Lamm-Riviere-2008}:

\medskip
\emph{"We expect similar theorems to remain true for general even order
elliptic systems of the type  \eqref{eq:Lamm-Riviere 2008}."}
\medskip

\noindent and the following expectation from a conference proceeding of Rivi\`ere \cite[Page 108]{Riviere-2011}:

\medskip
\emph{"It is natural to believe that a general result exists for $m$-th order linear
systems in $m$ dimension whose 1st order potential is antisymmetric."}
\medskip

In the very interesting recent work of de Longueville and Gastel \cite{deLongueville-Gastel-2019},  the following $2k$-th ($k\ge 3$) order linear elliptic system with  antisymmetric first order potential is considered:
\begin{equation}\label{eq: Longue-Gastel system}
\Delta^{k}u=\sum_{i=1}^{k-1}\Delta^{i}\left\langle V_{i},du\right\rangle +\sum_{i=0}^{k-2}\Delta^{i}\delta\left(w_{i}du\right)+\langle V_0, du\rangle,
\end{equation}
where the authors successfully extended the conservation law of Rivi\`ere \cite{Riviere-2007} and Lamm and  Rivi\`ere \cite{Lamm-Riviere-2008} to \eqref{eq: Longue-Gastel system}. As an application of their conservation law, they derived continuity of weak solutions; see also \cite{Horter-Lamm-2020} for an alternative approach.

We would like to point out that the system \eqref{eq: Longue-Gastel system} includes both extrinsically and intrinsically polyharmonic mappings. Recall that for $k\geq 2$ and $u\in W^{k,2}(B^n,N)$, the extrinsic \emph{$k$-harmonic energy functional} is defined as
\begin{equation*}\label{eq:k harmonic energy}
E(u):=\int_{B^n}|\nabla^k u|^2 dx
\end{equation*}
and weakly extrinsic $k$-polyharmonic mappings are critical points of  the $k$-harmonic energy with respect to compactly supported variations on $N$ (see e.g. \cite{Angelsberg-Pumberger-AGAG2009}). Similar to the fourth order elliptic system \eqref{eq:Lamm-Riviere 2008}, higher order elliptic geometric variational problems have attracted great attention in the recent literature; see e.g. \cite{Angelsberg-Pumberger-AGAG2009,deLongueville-Gastel-2019,Gastel-Scheven-2009CAG,
Goldstein-Strzelecki-Zatorska-2009,Horter-Lamm-2020,Lamm-Wang-2009} and the references therein for more along this direction. From an analytic point of view, one particular interest for such higher order elliptic systems is that they have critical growth nonlinearities and classical regularity method just fails to apply in such a borderline case. We would like to emphasize that the regularity theorems for $k$-polyharmonic mappings (into a closed manifold) were obtained by Angelsberg and Pumberger  \cite{Angelsberg-Pumberger-AGAG2009} under certain strong a priori regularity assumptions  for the solution and by  Gastel and Scheven  \cite{Gastel-Scheven-2009CAG} under natural regularity assumption via the construction of Coulomb frames (similar to the approach of  \cite{Wang-2004-CPAM}).


%
%
%
%
%
%

In view of the work of  Rivi\`ere and Struwe \cite{Riviere-Struve-2008},  it is natural to generalize Problem \ref{prob:Struwe} as follows: 	


\begin{problem}\label{prob:uniform} 	
Extend the   regularity theory of  Rivi\`ere and Struwe \cite{Riviere-Struve-2008} to  $2k$-th ($k\ge 2$) order linear elliptic systems in all  dimension $n$, $n\ge 2k$; particularly in the conformal case $n=2k$.
\end{problem}

The prime goal of this paper is to extend the techniques of Rivi\`ere and Struwe \cite{Riviere-Struve-2008} to higher order systems, so as  to provide a direct approach for the H\"older regularity of weak solutions of system \eqref{eq: Longue-Gastel system} in the critical dimension $2k$, and hence provide a partial positive answer to Problem \ref{prob:uniform}. It is also natural to explore the system \eqref{eq: Longue-Gastel system} in $B^n$, $n>2k$.  We shall  provide a discussion in this respect in the last section.


Fix $\ensuremath{2k \geq 4,m\in\mathbb{N}}$. We consider the linear elliptic system \eqref{eq: Longue-Gastel system} of de Longueville and Gastel \cite{deLongueville-Gastel-2019}, that is,
\[
\Delta^{k}u=\sum_{i=1}^{k-1}\Delta^{i}\left\langle V_{i},du\right\rangle +\sum_{i=0}^{k-2}\Delta^{i}\delta\left(w_{i}du\right)+\left\langle V_0,du\right\rangle \qquad \text{in } B^{2k}.
\]
As in  \cite{deLongueville-Gastel-2019},  the coefficient functions are assumed to   satisfy
\begin{equation}\label{eq:Assumption 1}
\begin{aligned} & w_{i}\in W^{2i+2-k,2}\left(B^{2k},\mathbb{R}^{m\times m}\right)\text{ for }i\in\{0,\ldots,k-2\},\\
 & V_{i}\in W^{2i+1-k,2}\left(B^{2k},\mathbb{R}^{m\times m}\otimes\wedge^{1}\mathbb{R}^{2k}\right)\text{ for }i\in\{1,\ldots,k-1\},
\end{aligned}
\end{equation}
and $$V_{0}=d\eta+F$$ with
\begin{equation}\label{eq:Assumption 2}
\eta\in W^{2-k,2}\left(B^{2k},\boldsymbol{so(m)}\right)\quad \text{  and  }\quad  F\in W^{2-k,\frac{2k}{k+1},1}\left(B^{2k},\mathbb{R}^{m\times m}\otimes\wedge^{1}\mathbb{R}^{2k}\right).
\end{equation}
The definition of relevant function spaces will be given in Section \ref{sec:relevant function spaces} below.

Our main result of this paper is the following.
\begin{theorem}\label{thm:general case}
Suppose the assumptions \eqref{eq:Assumption 1} and \eqref{eq:Assumption 2} are satisfied. Then every weak solution of \eqref{eq: Longue-Gastel system} is locally H\"older continuous in $B^{2k}$.
\end{theorem}

The idea for proving Theorem \ref{thm:general case} is inspired by Wang \cite{Wang-2004-CPAM} and our recent work \cite{Guo-Xiang-2019-Boundary}, and we benefit a lot from the recent work of de Longueville and Gastel \cite{deLongueville-Gastel-2019} regarding the derivation of higher order elliptic system \eqref{eq: Longue-Gastel system}\footnote{In our initial manuscript, we only considered the higher order systems up to order six.}. In spirit of \cite{Riviere-Struve-2008} and \cite{Struwe-2008}, we first apply a Gauge transform to get an equivalent form of system \eqref{eq: Longue-Gastel system}. Then, we use a standard potential theory to derive certain decay estimates in Lorentz spaces concerning derivatives of $u$ from the Gauge equivalent system. Finally, the H\"older continuity of $u$ follows from Morrey's Dirichlet growth theorem \cite{Giaquinta-Book}. Note that in both papers \cite{deLongueville-Gastel-2019,Horter-Lamm-2020}, the authors obtained only continuity but not H\"older continuity. In this sense, our result also provides an improvement. As a simple application of Theorem \ref{thm:general case}, we point out that every weakly (extrinsic or intrinsic) polyharmonic mapping from $B^{2k}$ into a closed Riemannian manifold is thus smooth by   Gastel and Scheven \cite[Theorem 1.2]{Gastel-Scheven-2009CAG}.

Generally speaking, the techniques used in our main proofs are not really new. The idea of using Gauge transform to rewrite the system was already used in \cite{Riviere-Struve-2008} and \cite{Struwe-2008}. The use of Riesz potential theory, together with the theory of Lorentz spaces, already appeared in many places; see for instance \cite{Wang-2004-MZ,Wang-2004-CPAM}. An ingredient which plays a crucial role in our approach  is the duality between the Lorentz spaces $L^{p,\infty}$ and $L^{p/(p-1),1}$ for $1<p<\wq$, which goes back to Bethuel \cite{Bethuel-1992} for second order systems. Nevertheless, a suitable combination of all these (relatively well-known) ideas leads to an affirmative solution of Problem \ref{prob:uniform} in critical dimensions.

Theorem \ref{thm:general case} is stated for the critical dimensions and it is a natural question to ask for the supercritical dimensions (note that the approach of Rivi\`ere and Struwe \cite{Riviere-Struve-2008}  works for supercritical dimensions as well). We believe our approach also applies in the supercritical dimensional case, but there is a serious technical difficulty from harmonic analysis preventing a straightforward extension to the supercritical dimensional case. We shall briefly discuss this difficulty in the final section.

The paper is organized as follows. In section \ref{sec: preliminaries} we present some knowledge on necessary function spaces and Riesz potential theory. In section \ref{sec: Proof of main results} we apply Riesz potential theory to prove Theorem \ref{thm:general case} in the fourth order case as a warm-up, and we hope this would make our exposition more friendly to the readers. Then we prove Theorem \ref{thm:general case} in the following Section \ref{sec: higher oder system}. In the last section, we discuss a possible approach to extend Theorem \ref{thm:general case} in supercritical dimensions.

\textbf{Acknowledgement.} Both authors are grateful to Prof.~T. Rivi\`ere, Prof.~M. Struwe and Prof.~T. Lamm for their interest in this work and for the pleasant discussions during the preparation of this work. They are deeply indebted to the anonymous referees for their extremely valuable comments and corrections that greatly improved  the exposition.

\section{Preliminaries and auxiliary results}\label{sec: preliminaries}

\subsection{Lorentz-Sobolev function spaces and related}\label{sec:relevant function spaces}

In this section, we recall the definition of Lorentz and Lorentz-Sobolev spaces and collect some basic facts that will be used later (we recommend the interested authors to read the references \cite{Adams-book} for more information about these spaces). Throughout this section, we will assume that $\Om$ is a domain in $\R^{n}$.

\subsubsection{Lorentz spaces}

For a measurable function $f\colon \Om\to\R$, denote by $\de_{f}(t)=|\{x\in\Om:|f(x)|>t\}|$
its distributional function and by $f^{\ast}(t)=\inf\{s>0:\de_{f}(s)\le t\}$,
$t\ge0$, the nonincreasing rearrangement of $|f|$. Define
\begin{eqnarray*}
	f^{\ast\ast}(t)\equiv\frac{1}{t}\int_{0}^{t}f^{\ast}(s)\D s, &  & t>0.
\end{eqnarray*}
The\emph{ Lorentz space} $L^{p,q}(\Om)$ ($1<p<\wq,1\le q\le\wq$)
is the space of measurable functions $f:\Om\to\R$ such that
\[
\|f\|_{L^{p,q}(\Om)}\equiv\begin{cases}
\left(\int_{0}^{\wq}(t^{1/p}f^{\ast\ast}(t))^{q}\frac{\D t}{t}\right)^{1/q}, & \text{if }1\le q<\wq,\\
\sup_{t>0}t^{1/p}f^{\ast\ast}(t) & \text{if }q=\wq
\end{cases}
\]
is finite. It is well-known that $L^{p,\infty}(\Omega)$ is just the weak $L^p$-space consisting of measurable functions $f$ in $\Om$ for which
$\sup_{t>0}t^p\left|\big\{x\in \Omega: |f(x)|> t\big \}\right|<\infty$ holds.   We shall also use $L^{1,\wq}$ to denote the weak $L^1$-space.

The following H\"older's inequality in Lorentz space is well-known.

\begin{proposition}[\cite{ONeil-1963}] \label{prop: Lorentz-Holder inequality}
Let $1<p_{1},p_{2}<\wq$ and $1\le q_{1},q_{2}\le\wq$ be such that
\begin{eqnarray*}
\frac{1}{p}=\frac{1}{p_{1}}+\frac{1}{p_{2}}\le1 & \text{and} & \frac{1}{q}=\frac{1}{q_{1}}+\frac{1}{q_{2}}\le1.
\end{eqnarray*}
Then, $f\in L^{p_{1},q_{1}}(\Om)$ and $g\in L^{p_{2},q_{2}}(\Om)$
implies $fg\in L^{p,q}(\Om)$. Moreover,
\[
\|fg\|_{L^{p,q}(\Om)}\le\|f\|_{L^{p_{1},q_{1}}(\Om)}\|g\|_{L^{p_{2},q_{2}}(\Om)}.
\]
\end{proposition}

\begin{proposition}[\cite{Ziemer}]\label{prop: Lorentz to Lorentz}
For $1<p<\wq$ and $1\le q_{1}\le q_{2}\le\wq$, there holds $L^{p}(\Om)=L^{p,p}(\Om)$
and $L^{p,q_{1}}(\Om)\subset L^{p,q_{2}}(\Om)$ with
\[
\|f\|_{L^{p,q_{2}}(\Om)}\le C(p,q_1,q_2)\|f\|_{L^{p,q_{1}}(\Om)}.
\]
Moreover, if $|\Om|<\wq$, then $L^{p,q}(\Om)\supset L^{r,s}(\Om)$
for $1<p<r<\wq$, $1\le q,s\le\wq$, and
\[
\|f\|_{L^{p,p}(\Om)}\le C_{r,p} |\Om|^{\frac{1}{p}-\frac{1}{r}}\|f\|_{L^{r,\wq}(\Om)}.
\]
\end{proposition}

The last inequality in Proposition~\ref{prop: Lorentz to Lorentz} implies that for all $B^{2k}_r\subset \R^{2k}$ and all $1\le p<2k$, there exists a constant $C=C(k,p)>0$, such that
\begin{equation}\label{eq: Lorentz-to-Morrey}
r^{p-2k}\int_{B_{r}^{2k}}|\na u|^{p}\le C\|\na u\|_{L^{2k,\infty}(B_{r}^{2k})}^{p}.
\end{equation}
In our later proofs, this will be combined with Morrey's Dirichlet growth theorem to derive H\"older continuity of a given solution.

\subsubsection{Lorentz-Sobolev spaces}

For $k\in \mathbb{N}$, $1\leq p, q\leq \infty$, the Lorentz-Sobolev space $W^{k,p,q}(\Omega)$ is defined to be the space of all functions $f\colon \Omega\to \R$ which are weakly differentiable up to order $k$ with $D^{\alpha}f\in L^{p,q}(\Omega)$ for all $|\alpha|\leq k$. The following fact will be used:
\begin{itemize}
 \item Let $\Om\subset\R^n$ be a bounded smooth domain, $k\in \N , 1<p<\wq, 1\le q<\wq$. There exists a bounded linear operator $E: W^{k,p,q}(\Omega)\to W^{k,p,q}(\R^n)$, such that if $f\in W^{k,p,q}(\Omega)$, then $Ef=f$ a.e. in $\Om$, and there exists a constant $C>0$ such that for all $f\in W^{k,p,q}(\Omega)$, there holds
   \[\|Ef\|_{W^{k,p,q}(\R^n)}\le C\|f\|_{W^{k,p,q}(\Omega)};\]  see for instance \cite[Theorem 2]{DeVore-Scherer-1979}, \cite[Page 19]{deLongueville-2018}.
\end{itemize}

We shall also need the Lorentz-Sobolev spaces with negative exponents. For $k\in \N$, $1\leq p,q\leq \infty$, the Lorentz-Sobolev space $W^{-k,p,q}(\Omega)$ is defined to be the space of all distributions $f$ on $\Om$ of the form $f=\sum_{|\alpha|\leq k}D^{\alpha}f_{\alpha}$ with $f_\alpha\in L^{p,q}(\Om)$. The corresponding norm is defined as
\[ \|f\|_{W^{k,p,q}(\Omega)}:=\inf \sum_{|\alpha|\leq k}\|f_\alpha\|_{L^{p,q}(\Omega)},\]
where the infimum is taken over all decompositions of $f$ as given in the definition.

Let $\Omega\subset \R^n$ be a bounded smooth domain. The following facts about Lorentz-Sobolev spaces with negative exponents can be found in \cite{deLongueville-2018,deLongueville-Gastel-2019}.
\begin{itemize}
\item (Generalized H\"older's inequality) Suppose $f\in W^{-k_0,p_0,q_0}(\Omega)$ and $g\in W^{k_1,p_1,q_1}(\Omega)$ with $k_0,k_1\in \N$, $1<p_0,p_1<\infty$ and $1\leq q_0,q_1<\infty$. If $k_0\le k_1$, $\frac{1}{p_0}+\frac{1}{p_1}\leq 1$ and $k_1p_1<n$, then $fg\in W^{-k_0,s,t}(\Omega)$ with $s=\frac{np_0p_1}{n(p_0+p_1)-k_1p_0p_1}$ and $\frac{1}{t}=\min\big\{1,\frac{1}{q_0}+\frac{1}{q_1}\big\}$. Moreover,
\begin{equation}\label{eq:product norm for Lorentz-Sobolev}
\|fg\|_{W^{-k_0,s,t}(\Om)}\leq C\|f\|_{W^{-k_0,p_0,q_0}(\Om)}\|g\|_{W^{k_1,p_1,q_1}(\Om)}.
\end{equation}
The assertion continues to hold in the case $k_1p_1=n$ if we additionally assume $g\in L^\wq$.

\item (Sobolev embeddings of Lorentz-Sobolev functions with negative exponent) If $k\in \mathbb{Z}$, $l\in \N$, $1<p<\frac{n}{l}$ and $1\leq q\leq \infty$, then every $f\in W^{k,p,q}(\Omega)$ is also in $W^{k-l,\frac{np}{n-lp},q}(\Omega)$ with
$$\|f\|_{W^{k-l,\frac{np}{n-lp},q}(\Omega)}\leq C\|f\|_{W^{k,p,q}(\Omega)}.$$
\end{itemize}


\subsection{Fractional Riesz operators}

Let $I_\alpha$ be the standard fractional Riesz operator, that is, the operator whose convolution kernel is $|x|^{\alpha-n}$, $x\in \R^n$. The following estimates on fractional Riesz operators between Lorentz spaces are well-known.

\begin{proposition}\label{prop:Riesz potential}
For $0<\alpha<n$, $1<p<n/\al$, $1\le q\leq q'\le\wq$, the fractional Riesz operators
\[
I_{\al}\colon L^{p,q}(\R^{n})\to L^{\frac{np}{n-\al p},q'}(\R^{n})
\]
and
\[
I_{\al}\colon L^{1}(\R^{n})\to L^{\frac{n}{n-\al},\wq}(\R^{n})
\]
are bounded.
\end{proposition}
In the later proof, we will mainly use the special case $q'=\wq$.

\section{Warm up: the fourth order elliptic system}\label{sec: Proof of main results}

In this section, we shall address Problem \ref{prob:Struwe} and demonstrate the general scheme of our approach to establish interior H\"older regularity of solutions without using conservation law. The proof of Theorem \ref{thm:general case} in the next section shall be completely similar to what we presented in this section, but of course with much heavier computation.

Let $B_{10}^4=\{x\in\R^{4}:|x|<10\}$ and $u\in W^{2,2}(B_{10}^4,\R^{m})$.
Consider the following fourth order linear elliptic system of Struwe \cite{Struwe-2008}
\begin{eqnarray}
\De^{2}u=\De(D\cdot\na u)+{\rm div}(E\cdot \na u)+F\cdot\na u &  & \text{in }B_{10}^4\label{eq: Lamm-Riviere system n=4}
\end{eqnarray}
with
\[
F=\De\Om+G.
\]
Here, we make the following regularity assumptions:
\begin{eqnarray}
D\in W^{1,2}(B_{10}^4, M_m\otimes \Lambda^1\R^4), & E\in L^{2}(B_{10}^4,M_m\otimes\Lambda^{1}\R^{4}),\label{eq: coefficients D-E n=4}
\end{eqnarray}
and
\begin{equation}
G\in L^{\frac{4}{3},1}(B_{10}^4,M_{m}\otimes\Lambda^{1}\R^{4}),\label{eq: coefficient G n=4}
\end{equation}
\begin{equation}
\Om\in W^{1,2}(B_{10}^4,so_{m}\otimes\Lambda^{1}\R^{4}).\label{eq: coefficients Omega n=4}
\end{equation}
In Section \ref{subsec:regularity assumptions}  we shall see that these regularity assumptions are satisfied  when $u$ is a stationary biharmonic mapping.

Our first main result gives an affirmative answer to Problem \ref{prob:Struwe} in the conformal case $n=4$ for the system \eqref{eq: Lamm-Riviere system n=4}.
\begin{theorem}\label{thm:main thm}
Suppose $D,E,G$ and $\Om$ satisfy the regularity assumptions \eqref{eq: coefficients D-E n=4}, \eqref{eq: coefficient G n=4}, \eqref{eq: coefficients Omega n=4}. Then every
 weak solution of system \eqref{eq: Lamm-Riviere system n=4} in $B^4_{10}$
	is locally  H\"older continuous.
\end{theorem}

\subsection{Gauge-equivalent form of \eqref{eq: Lamm-Riviere system n=4}}\label{subsec:Gauge transform}
Following \cite{Struwe-2008}, we shall use the theory of Gauge transform to deal with system \eqref{eq: Lamm-Riviere system n=4}. The original work on Gauge transform in Sobolev spaces was due to Uhlenbeck \cite{Uhlenbeck-1982} and the results were extended to suitable Morrey spaces by Meyer and Rivi\`ere \cite{Meyer-Riviere-2003}, and Tao and Tian \cite{Tao-Tian-2004} in the study of Yang-Mills fields. We will use the following Gauge transform in Sobolev spaces; see \cite[Lemma 3.3]{Struwe-2008} or \cite[Theorem A.5]{Lamm-Riviere-2008} for more general statements and proofs.

\begin{lemma}\label{lem: Gauge transform} There exist $\ep=\ep(m)>0$
and $C=C(m)>0$ with the following property: For every $\Om\in W^{1,2}(B,so_{m}\otimes\wedge^{1}\R^{4})$ $(B=B^4_{10})$
with
\[
\|\Om\|_{W^{1,2}(B)}\le\ep,
\]
there exist $P\in W^{2,2}(B,SO_{m})$ and $\xi\in W^{2,2}(B,so_{m}\otimes\wedge^{2}\R^{4})$
such that
\begin{eqnarray*}
P\D P^{-1}+P\Om P^{-1}=\ast\D\xi &  & \text{in }B,
\end{eqnarray*}
and
\begin{eqnarray*}
\D\ast\xi=0\quad\text{in }B, &  & \xi=0\quad\text{on }\pa B.
\end{eqnarray*}
 Moreover, we have
\[
\begin{aligned} & \|\na^{2}P\|_{L^{2}(B)}+\|\na P\|_{L^{4}(B)}+\|\na^{2}\xi\|_{L^{2}(B)}+\|\na\xi\|_{L^{4}(B)}\\
 & \quad\le C\left(\|\na\Om\|_{L^{2}(B)}+\|\Om\|_{L^{4}(B)}\right)\le C\ep.
\end{aligned}
\]
\end{lemma}

Let $P,\xi$ be given by Lemma \ref{lem: Gauge transform}. By the computation in \cite{Struwe-2008} (see in particular \cite[Equation (35)]{Struwe-2008}), the gauge-equivalent form of \eqref{eq: Lamm-Riviere system n=4} is given by
\begin{equation}\label{eq: (35)}
\De(P\De u)={\rm div}^{2}(D_{P}\otimes\na u)+{\rm div}(E_{P}\cdot\na u)+G_{P}\cdot\na u+\ast\D\De\xi\cdot P\na u,
\end{equation}
where
\begin{equation}
(D_{P})_{\al}^{ik}=\de_{\al\be}P^{ij}D_{\be}^{jk}+2\pa_{\al}P^{ik},\label{eq: DP}
\end{equation}
\begin{equation}
(E_{P})_{\al\be}^{ik}=P^{ij}E_{\al\be}^{jk}-2\pa_{\al}P^{ij}D_{\be}^{jk}-\de_{\al\be}\De P^{ik}-2\pa_{\al\be}P^{ik}\label{eq: EP}
\end{equation}
and
\begin{equation}
\begin{aligned}G_{P} & =\na\De P+P\De\Om-(\De\ast\D\xi)P+\De PD+PG-\na P\cdot E\\
& =-2\na P\cdot\na\Om-\De P\Om+2\na(\ast\D\xi)\cdot\na P+\ast\D\xi\De P+\De PD+PG-\na P\cdot E,
\end{aligned}
\label{eq: GP}
\end{equation}
where the second line is obtained by taking Laplacian on the Gauge
decomposition. The second line is meaningful since it concerns at
most second order derivatives.

Note that the Lorentz-Sobolev embedding implies that we have the following improved regularity:
\[
\na u,\na P,\na\xi,\Om\in L^{4,2}.
\]
From the expressions of $D_P$, $E_P$ and $G_P$, we easily deduce the following estimates:
\begin{itemize}
	\item[i)] $|D_{P}|\le|D|+|\na P|$;
	\item[ii)] $|E_{P}|\lesssim|E|+|\na P||D|+|\na^{2}P|$;
	\item[iii)] $|G_{P}|\lesssim|\na P||\na\Om|+|\na^{2}P||\Om|+|\na^{2}\xi||\na P|+|\na\xi||\na^{2}P|+|\na^{2}P||D|+|G|+|\na P||E|$.
\end{itemize}

a) Since $D\in W^{1,2}$, by the Lorentz-Sobolev embedding,
$
D\in L^{4,2}.
$
This in turn implies that
\[
D_{P}\in L^{4,2}(B_{10}^4).
\]

b) Since $E\in L^{2}$, ii) implies that
\[
E_{P}\in L^{2}.
\]

c)
Applying Proposition \ref{prop: Lorentz-Holder inequality} repeatedly for each term in iii) and note that $G\in L^{\frac{4}{3},1}$, we easily conclude
\[
G_{P}\in L^{\frac{4}{3},1}(B_{10}^4).
\]

\subsection{Proof of Theorem \ref{thm:main thm}}

Now we can prove Theorem \ref{thm:main thm}. From the above subsections we deduce  that
\[
\na u,\na P,\na\xi,\Om\in L^{4,2},
\]
\[
D_{P}\in L^{4,2},\quad  E_{P}\in L^{2}=L^{2,2}\quad \text{and}\quad   G_{P}\in L^{\frac{4}{3},1}.
\]
Let $\epsilon=\epsilon(m)$ be given as in Lemma \ref{lem: Gauge transform} and we may additionally assume that
$$\|u\|_{W^{2,2}}+\|D\|_{W^{1,2}}+\|E\|_{L^2}+\|G\|_{L^{\frac{4}{3},1}}+\|\Omega\|_{W^{1,2}}<\epsilon.$$
All the previous coefficients functions are defined on the ball $B^4_{10}$, with norms bounded by $C\ep$. We first extend them from $B^4_{5}$  to $\R^{4}$ such that their respective (Lorentz-Sobolev) norms in $\R^4$ are no more than a bounded constant multiplying the corresponding norm in $B^4_{5}$.  For simplicity of notations, we keep using the same notation for the extended coefficient functions.

\textbf{Step 1.} Decompose $P\nabla u$ via Hodge decomposition.

By the Hodge decomposition (see e.g. \cite[Page 434]{Wang-2004-CPAM}), there are a function $f$ and a co-closed $2$-form $g$ such that $\nabla f,\nabla g\in L^{4}(\R^4)$ and
\[
P du=df+\ast dg
\]
in $\R^{4}$. It follows that $f$ satisfies
\[
\De^{2}f=\De{\rm div}(P\na u)=\De(P\De u)+\De(\na P\cdot\na u),
\]
or equivalently,
\[
\De^{2}f={\rm div}^{2}(D_{P}\otimes\na u)+{\rm div}(E_{P}\cdot\na u)+G_{P}\cdot\na u+\ast\D\De\xi\cdot P\na u+\De(\na P\cdot\na u).
\]
Similarly, $g$ satisfies
\[
\De^{2}g=\Delta\big(\ast (dP\wedge\D u)\big).
\]

\textbf{Step 2.} Separate the biharmonic part of $f$.
\medskip

Let
$
I_{4}=c\log (|x|)
$
be the fundamental solution of $\De^{2}$ in $\R^{4}$ and set
\[
\begin{aligned}
\tilde{f}&=I_{4}\ast\left({\rm div}^{2}(D_{P}\otimes\na u)+{\rm div}(E_{P}\cdot\na u)+G_{P}\cdot\na u+\ast\D\De\xi\cdot P\na u+\De(\na P\cdot\na u)\right)\\
&=:I_4\ast K.
\end{aligned}
\]
It is clear that
\[
\De^{2}(f-\tilde{f})=0.
\]

\medskip
\textbf{Step 3.} Estimate $\|\D\tilde{f}\|_{L^{4,\infty}(\R^{4})}$.
\medskip

Recall that $I_{\al}(x)=|x|^{\al-4}$, $0<\al<4$ are the fractional Riesz potentials.

(1) Set $J_{1}=I_{4}\ast{\rm div}^{2}(D_{P}\otimes\na u)$. Integration by parts implies
\[
|\nabla J_{1}|\lesssim|I_{1}(|D_{P}||\na u|)|.
\]
Since $D_{P}\in L^{4,2}$ and $\na u\in L^{4,2}\subset L^{4,\wq}$,
we may apply Proposition \ref{prop: Lorentz-Holder inequality}
to deduce
\[
|D_{P}||\na u|\in L^{2}.
\]
Applying Proposition \ref{prop:Riesz potential} with $p=2$, $\al=1$ and $n=4$, we infer that $\na J_{1}\in L^{4}$.
Furthermore, we have
\begin{equation}
\begin{aligned}
\|\na J_{1}\|_{L^{4,\infty}}&\le\|\na J_{1}\|_{L^{4}}\lesssim \|I_1(|D_P||\nabla u|)\|_{L^{4}}\stackrel{\text{Prop }\ref{prop:Riesz potential}}{\lesssim} \||D_P||\nabla u|\|_{L^{2}}\\
&\stackrel{\text{Prop }\ref{prop: Lorentz-Holder inequality}}{\leq} \|D_{P}\|_{L^{4,2}}\|\nabla u\|_{L^{4,\infty}(\R^{4})}\lesssim\ep\|\nabla u\|_{L^{4,\infty}(B_{1})}.
\end{aligned}
\label{eq:Estimate of J1}
\end{equation}

(2) Set $J_{2}=I_{4}\ast{\rm div}(E_{P}\cdot\na u)$. Then
\[
|\na J_{2}|\lesssim I_{2}(|E_{P}||\na u|).
\]
Since $E_{P}\in L^{2}$ and $\na u\in L^{4,2}$, we apply Proposition \ref{prop: Lorentz-Holder inequality} (with $p_{1}=2,p_{2}=4$ and
$q_{1}=2,q_{2}=\wq$, then $p=4/3,q=2$) to get
\[
E_{P}\cdot\na u\in L^{\frac{4}{3},2}.
\]

Since $I_2\colon L^{\frac{4}{3},2}\to L^{4,\infty}$ is bounded by Proposition \ref{prop:Riesz potential}, we estimate as follows:
\begin{equation}
\begin{aligned}
\|\na J_{2}\|_{L^{4,\infty}}&\lesssim \|I_{2}(|E_{P}||\na u|)\|_{L^{4,\wq}}\lesssim\|E_{P}\cdot\na u\|_{L^{\frac{4}{3},2}}\\
&\stackrel{\text{Prop }\ref{prop: Lorentz-Holder inequality}}{\le}\|E_{P}\|_{L^{2}}\|\nabla u\|_{L^{4,\infty}(\R^{4})}\lesssim\ep\|\nabla u\|_{L^{4,\infty}(B_{1})}.
\end{aligned}
\label{eq: estimate of j2}
\end{equation}

(3) Set $J_{3}=I_{4}\ast(G_{P}\cdot\na u)$. Then,
\[
|\na J_{3}|\le I_{3}(|G_{P}\cdot\na u|).
\]
Since $G_{P}\in L^{\frac{4}{3},1}$ and $\na u\in L^{4,\wq}$,
we know from Proposition \ref{prop: Lorentz-Holder inequality} that
\[
G_{P}\cdot\na u\in L^{1}.
\]
Thus, we may apply Proposition \ref{prop:Riesz potential} with $n=4$ and $\al=3$ to infer that $\na J_{3}\in L^{4,\infty}$
and
\begin{equation}
\begin{aligned}
\|\na J_{3}\|_{L^{4,\infty}}&\lesssim \|I_{3}(|G_{P}\cdot\na u|) \|_{L^{4,\infty}} \lesssim \|G_{P}\cdot\na u\|_{L^{1}} \\
&\stackrel{\text{Prop }\ref{prop: Lorentz-Holder inequality}}{\leq }\|G_{P}\|_{L^{4/3,1}}\|\nabla u\|_{L^{4,\infty}(\R^{4})}\lesssim\ep\|\nabla u\|_{L^{4,\infty}(B_{1})}.
\end{aligned}
\label{eq: Estimate of J3}
\end{equation}

(4) Set $J_{4}=I_{4}\ast\left(\ast\D\De\xi\cdot P\na u\right)$. Up to a sign, we have
\[
J_{4}=\na I_{4}\ast\left(\De\xi\wedge P\D u\right)+I_{4}\left(\De\xi\wedge\D P\wedge\D u\right).
\]
It follows
\[
|\na J_{4}|\lesssim I_{2}(|\na^{2}\xi||\na u|)+I_{3}(|\na^{2}\xi||\na u||\na P|)=:J_{41}+J_{42}.
\]
Since $\na^{2}\xi\in L^{2}$ and $\nabla u\in L^{4,\infty}$, we know $|\na^{2}\xi||\na u|\in L^{\frac{4}{3},2}$.
Since $I_2\colon L^{\frac{4}{3},2}\to L^{4,\infty}$ is bounded, we estimate as follows:
\begin{equation*}
\begin{aligned}
\|\na J_{41}\|_{L^{4,\infty}}&\lesssim \|I_{2}(|\nabla^2 \xi||\na u|)\|_{L^{4,\wq}}\lesssim\||\nabla^2 \xi||\na u|\|_{L^{\frac{4}{3},2}}\\
&\stackrel{\text{Prop }\ref{prop: Lorentz-Holder inequality}}{\le}\|\nabla^2 \xi\|_{L^{2}}\|\nabla u\|_{L_{\ast}^{4}(\R^{4})}\lesssim\ep\|\nabla u\|_{L^{4,\infty}(B_{1})}.
\end{aligned}
\end{equation*}
Similarly, we have $|\na^{2}\xi||\na u||\na P|\in L^{1}$, and we may estimate as in \eqref{eq: Estimate of J3} to deduce
\[
\|J_{42}\|_{L^{4,\infty}}\lesssim\ep\|\nabla u\|_{L^{4,\infty}(B_{1})}.
\]
Combining the above two estimates yields
\begin{equation}
\|\na J_{4}\|_{L^{4,\infty}}\lesssim\ep\|\nabla u\|_{L^{4,\infty}(B_{1})}.\label{eq: Estimate of J4}
\end{equation}

(5) Set $J_{5}=I_{4}\ast(\De(\na P\cdot\na u))$. Then,
\[
|\na J_{5}|\le I_{1}(|\na u||\na P|).
\]
We may estimate as in \eqref{eq:Estimate of J1} to derive
\begin{equation}
\|\na J_{5}\|_{L^{4,\infty}}\lesssim\ep\|\nabla u\|_{L^{4,\infty}(B_{1})}\label{eq: Estimate of J5}
\end{equation}
Combining \eqref{eq:Estimate of J1}, \eqref{eq: estimate of j2},
\eqref{eq: Estimate of J3}, \eqref{eq: Estimate of J4} and \eqref{eq: Estimate of J5} all together, we finally obtain
\begin{equation}
\|\na\tilde{f}\|_{L^{4,\infty}(\R^{4})}\lesssim\ep\|\nabla u\|_{L^{4,\infty}(B_{1})}.\label{eq: error estimate 3}
\end{equation}

\textbf{Step 4.} Estimate $\|\ast dg\|_{L^{4,\infty}}$.

Note that $\De^{2}g=\De\left(\ast \big(dP\wedge du \big)\right)$. Thus, $g=I_{4}\ast \Big(\De\Big(\ast \big(dP\wedge du \big)\Big)\Big)$. This implies
\[
|dg|\lesssim I_{1}(|\na u||\na P|).
\]
Hence, we may estimate as in \eqref{eq: Estimate of J5} to deduce
\begin{equation}
\|*dg\|_{L^{4,\infty}}\lesssim\ep\|\nabla u\|_{L^{4,\infty}(B_{1})}\label{eq: Estimate of g}
\end{equation}

\textbf{Step 5.} Completion of the proof.

Since $P\D u=\D f+\ast dg=\D\tilde{f}+\ast dg+h$, where $h$ is a biharmonic 1-form, we have, for any $\ga\in(0,1)$,
\[
\begin{aligned}\|\nabla u\|_{L^{4,\infty}(B_{\ga})} & \le\|h\|_{L^{4,\infty}(B_{\ga})}+\|\nabla\tilde{f}\|_{L^{4,\infty}(B_{\ga})}+\|\ast dg\|_{L^{4,\infty}(B_{\ga})}\\
 & \lesssim\ga\|h\|_{L^{4,\infty}(B_{1})}+\|\nabla\tilde{f}\|_{L^{4,\infty}(B_{1})}+\|\ast dg\|_{L^{4,\infty}(B_{1})}\\
 & \lesssim\ga\|\na u\|_{L^{4,\infty}(B_{1})}+\ep\|\na u\|_{L^{4,\infty}(B_{1})},
\end{aligned}
\]
where we have used \eqref{eq: error estimate 3} and \eqref{eq: Estimate of g} in the last line. In conclusion, we have obtained
\begin{equation}
\|\nabla u\|_{L^{4,\infty}(B_{\ga})}\le C\left(\ga+\ep\right)\|\nabla u\|_{L^{4,\infty}(B_{1})}.\label{eq: Decay estimate 2}
\end{equation}
We thus conclude, by choosing $\ga+\ep$ sufficiently small and proceeding a standard iteration, that there exist $\alpha_0\in (0,1)$ and $r_0<\frac{1}{2}$ such that  for all $r<r_0$,
$$\|\nabla u\|_{L^{4,\infty}(B_r)}\leq Cr^{\alpha}.$$
Note that by \eqref{eq: Lorentz-to-Morrey} we have for any $p\in (1,4)$,
$$\big(r^{p-4}\int_{B_r}|\nabla u|^pdx\big)^{\frac{1}{p}}\leq \|\nabla u\|_{L^{4,\infty}(B_r)}.$$
Combining the above two estimates with Morrey's Dirichlet growth theorem implies the H\"older continuity of $u$.

\section{Higher order elliptic systems}\label{sec: higher oder system}

In this section, we shall prove our main result, Theorem \ref{thm:general case}. The approach is basically the same  as  that  of Section \ref{sec: Proof of main results}.  Throughout this section, we assume $k\ge 2$, and we use $B_r(x)$ to denote a ball  in $\R^{2k}$ centered at $x$  with radius $r$. Our aim is to  prove the following decay estimates in Lorentz spaces.

\begin{lemma}\label{lem: decay of Long-Gastel} Let $u\in W^{k,2}(B_{1}(0),\R^{m})$
be a solution to system (\ref{eq: Longue-Gastel system}). There exists
$\tau \in(0,1)$, depending only on $k$ and $m$, satisfying the following
property: for any $x\in B_{1/2}(0)$ and any $0<r<1/2$, there holds
\begin{equation}
\sum_{j=1}^{k}\|\na^{j}u\|_{L^{2k/j,\wq}(B_{\tau r}(x))}\le\frac 12 \sum_{j=1}^{k}\|\na^{j}u\|_{L^{2k/j,\wq}(B_{r}(x))}.\label{eq: decay estimates of higher order system}
\end{equation}
As a consequence, there exist $\al\in(0,1)$ and $C>0$, depending only
on $k$ and $m$, such that
\[
\sum_{j=1}^{k}\|\na^{j}u\|_{L^{2k/j,\wq}(B_{r}(x))}\le C\|u\|_{W^{k,2}(B_{1})}r^{\al}.
\]
 for all $x\in B_{1/2}$ and any $0<r<1/2$.\end{lemma}

Once this lemma is proved, the main result follows easily.

\begin{proof}[Proof of Theorem \ref{thm:general case}] By the inequality (\ref{eq: Lorentz-to-Morrey}) and the above lemma, we obtain,
for any $1\le p<2k$, that
\[
r^{p-2k}\int_{B_{r}(x)}|\na u|^{p}\le C\|\na u\|_{L^{2k,\wq}(B_{r}(x))}^{p}\le C\|u\|_{W^{k,2}(B_{1})}^{p}r^{p\al}
\]
for all $x\in B_{1/2}$ and any $0<r<1/2$. Then Morrey's Dirichlet
growth theorem implies that $u\in C^{0,\al}(B_{1/2})$. The proof
is complete. \end{proof}

In the following sections, we shall prove Lemma \ref{lem: decay of Long-Gastel}.
Firstly, we use Uhlenbeck's gauge transformation to rewrite \eqref{eq: Longue-Gastel system} into a gauge-equivalent form. Then we apply  Riesz potential theory to estimate the singular part of the given solution. Finally we prove Lemma \ref{lem: decay of Long-Gastel} by the similar techniques as in Section \ref{sec: Proof of main results}.

\subsection{Gauge-equivalent form of \eqref{eq: Longue-Gastel system}}
We shall need the following Uhlenbeck's gauge theorem from \cite[Theorem 2.4]{deLongueville-Gastel-2019}.
\begin{lemma}\label{lem: Gauge transform higher} Let $k,m\ge 2$.  There exist $\ep=\ep(k,m)>0$
and $C=C(k,m)>0$ satisfying the following property: For every $\Om\in W^{k-1,2}(B_r,so_{m}\otimes\wedge^{1}\R^{2k})$
with
\[
\|\Om\|_{W^{k-1,2}(B_r)}\le\ep,
\]
there exist $P\in W^{k,2}(B_{r/2},SO_{m})$ and $\xi\in W^{k,2}(B_{r/2},so_{m}\otimes\wedge^{2}\R^{2k})$
such that
\begin{eqnarray*}
P\D P^{-1}+P\Om P^{-1}=\ast\D\xi &  & \text{in }B_{r/2}.
\end{eqnarray*}
Moreover,
\[
\begin{aligned} & \|\D P\|_{W^{k-1,2}(B_{r/2})}+\|\delta \xi\|_{W^{k-1,2}(B_{r/2})}\le C\|\Omega\|_{W^{k-1,2}(B_{r})}\le C\ep.
\end{aligned}
\]
\end{lemma}


In the following, we shall write $B_r=B_r(0)$. Let $u\in W^{k,2}(B_1, \R^m)$ be a solution to \eqref{eq: Longue-Gastel system}. Denote
\begin{equation}\label{notation: theta_0}
\begin{aligned}
\ta_0&\equiv \sum_{i=0}^{k-2}\|w_i\|_{W^{2i+2-k,2}(B_1)}
+\sum_{i=1}^{k-1}\|V_i\|_{W^{2i+1-k,2}(B_1)}
 +\|\eta\|_{W^{2-k,2}(B_1)}+\|F\|_{W^{2-k,\frac{2k}{k+1},1}(B_1)}.
\end{aligned}
\end{equation}
Following \cite[Proof of Theorem 4.1]{deLongueville-Gastel-2019}, we may find $\Omega\in W^{k-1,2}(B_1,so(m)\otimes \wedge^1\R^{2k})$ satisfying
\[
-\De^{k-2}\delta\Om=\eta.
\]
If $\ta_0$ is small enough, then $\Om$ satisfies the smallness condition. Thus,
applying Lemma \ref{lem: Gauge transform higher} to $\Om$, we can find $P$ and $\xi$ which satisfy the properties stated in the lemma. 

Similar to the fourth order case, we want to derive the equation of
$P\De u$ in $B_{1/2}$.  Applying the Leibniz rule repeatedly, we find
\[
\De^{k-1}(P\De u)=P\De^{k}u+J+\na\De^{k-1}P\cdot\na u,
\]
where $J$ is given by
\begin{equation}
J\equiv\text{ a linear combination of terms of the form }{\rm div}^{a}\De^{b}(\na^{i}P\na^{j}u),\label{eq: def of J}
\end{equation}
where $a,b,i,j\in\N\cup\{0\}$, $1\le i,j\le k,$ $i+j\le2k-1$ and
$1\le a+2b=2k-i-j\le2k-2$.

Thus, by the system \eqref{eq: Longue-Gastel system}, we have
\begin{equation}
\De^{k-1}(P\De u)=J+\na\De^{k-1}P\cdot\na u+P\left(\sum_{l=0}^{k-1}\Delta^{l}\left\langle V_{l},du\right\rangle +\sum_{l=0}^{k-2}\Delta^{l}\delta\left(w_{l}du\right)\right)\quad \text{ in } B_{1/2}.
\label{eq: First time rewrite equation}
\end{equation}
It remains to deal with the last term in \eqref{eq: First time rewrite equation}.

As we have seen from the previous section, when  applying  Riesz potential theories to derive decay estimates,  it is not important how big the  constant coefficients  in front of terms of \eqref{eq: First time rewrite equation}.  From now on, we shall use $\sum_i f_i$ to represent linear combinations of $f_i$ without mentioning the constant coefficients of $f_i$.

For the term involving $V_{0}$, we  compute as that of de Longueville-Gastel \cite[the last line of page 12]{deLongueville-Gastel-2019} to deduce
\[
PV_{0}\cdot\na u=-\na\De^{k-1}P\cdot\na u+K_{0}\cdot\na u,
\]
where $K_{0}\in W^{2-k,\frac{2k}{k+1},1}$ (note  that our $K_0$ absorbs an extra $P$ comparing with \cite{{deLongueville-Gastel-2019}}).

For terms involving $V_{l}$ ($1\le l\le k-1)$, we consider two cases separately:
\begin{description}
\item[Case I] $2l\le k-1$;
\item[Case II] $2l>k-1$.
\end{description}

In \textbf{Case I}, we have
\[
P\De^{l}(V_{l}\cdot\na u)=\sum_{i=0}^{2l}P\na^{i}V_{l}\na^{2l+1-i}u,
\]
where with an abuse of notation we denote $\nabla^i$ for differential operators of order $i$, that is $\nabla^i$ is of the form $\nabla^{i_0}\Delta^{j_0}$ with $i=i_0+2j_0$.
Since $P\in W^{k,2}\cap L^{\wq}$ and $V_{l}\in W^{2l+1-k,2}$, we
know that $P\na^{i}V_{l}\in W^{2l+1-k-i,2,1}$ with $1-k\le2l+1-k-i\le0$.
Thus,
\[
\sum_{l=0}^{(k-1)/2}P\De^{l}(V_{l}\cdot\na u)=\sum_{l=0}^{(k-1)/2}\sum_{i=0}^{2l}K_{i,l}\na^{2l+1-i}u
\]
for some functions $K_{i,l}\in W^{2l+1-i-k,2,1}$. Organizing the
summation according to the order of the derivatives of $u$, we can
write the summation as
\[
\sum_{l=0}^{(k-1)/2}P\De^{l}(V_{l}\cdot\na u)=\sum_{i=1}^{k}K_{i}\na^{i}u
\]
for some $K_{i}\in W^{i-k,2,1}$, $1\le i\le k$. Note that $K_{i}, (1\le i\le k)$
are Sobolev functions with negative exponent. To make the summation
clear, write $K_{i}=\sum_{|\al|\le k-i}\na^{\al}K_{i,\al}$ for some
$K_{i,\al}\in L^{2,1}$. Then
\[
\sum_{i=1}^{k}K_{i}\na^{i}u=\sum_{i=1}^{k}\sum_{|\al|=0}^{k-i}\sum_{0\le\be\le\al}\na^{\be}(K_{i,\al}\na^{\al+i-\be}u).
\]
Again, organizing the summation according to the order of the derivatives
of $u$, we obtain
\[
\sum_{i=1}^{k}K_{i}\na^{i}u=\sum_{a=0}^{k-1}\na^{a}\left(\sum_{i=1}^{k-a}K_{i,a}\na^{i}u\right)
\]
for some functions $K_{i,a}\in L^{2,1}$. Thus,
\begin{equation}
\sum_{l=0}^{(k-1)/2}P\De^{l}(V_{l}\cdot\na u)=\sum_{a=0}^{k-1}\na^{a}\left(\sum_{i=1}^{k-a}K_{i,a}\na^{i}u\right)\qquad  \text{ in } B_{1/2},\label{eq: case 1}
\end{equation}
with coefficient functions $K_{i,a}\in L^{2,1}(B_{1/2})$ for all $i,a$ such that  $1\le i\le k-a\le k$.

In \textbf{Case II}, $0<2l+1-k\le k-1$.
Note that
\begin{equation}\label{eq:iteration order k}
\begin{aligned}
P\nabla^kf&=\nabla\big(P\nabla^{k-1}f\big)-\nabla p\cdot \nabla^{k-1}f\\
&=\left[\nabla^2\big(P\nabla^{k-2}f\big)-\nabla\big(\nabla P\nabla^{k-2}f \big)\right]-\left[\nabla\big(\nabla P\nabla^{k-2}f-\nabla^2P\nabla^{k-2}f \big)\right]\\
&=\sum_{s=0}^2\nabla^{2-s}\big(\nabla^s P  \nabla^{k-2}f\big)=\cdots\cdots=\sum_{s=0}^k\nabla^{k-s}\big(\nabla^sP f\big).
\end{aligned}
\end{equation}
Thus, we obtain
\[
\begin{aligned}P\De^{l}(V_{l}\cdot\na u) & =P\na^{k-1}\na^{2l+1-k}(V_{l}\cdot\na u)\\
 & =P\na^{k-1}\left(\sum_{t=0}^{2l+1-k}\na^{t}V_{l}\na^{2l+2-k-t}u\right)\\
 & =\sum_{a=0}^{k-1}\na^{k-1-a}\left[\na^{a}P\left(\sum_{t=0}^{2l+1-k}\na^{t}V_{l}\na^{2l+2-k-t}u\right)\right]\\
 & =\sum_{a=0}^{k-1}\sum_{t=0}^{2l+1-k}\na^{k-1-a}\left(\na^{a}P\na^{t}V_{l}\na^{2l+2-k-t}u\right).
\end{aligned}
\]
Note that the right hand side of the above equality are of divergence
structure, except the term where $a=k-1$. Moreover,
\[
\na^{a}P\na^{t}V_{l}\in W^{k-a,2}\cdot W^{2l+1-k-t,2}\hookrightarrow L^{\frac{2k}{a},2}\cdot L^{\frac{2k}{2k-(2l+1-t)},2}\hookrightarrow L^{\frac{2k}{2k-(2l+1-t)+a},1}
\]
and
\[
\na^{2l+2-k-t}u\in L^{\frac{2k}{2l+2-t-k},\wq},
\]
which imply, by H\"older's inequality, that
\[
\na^{a}P\na^{t}V_{l}\na^{2l+2-k-t}u\in L^{\frac{2k}{k+a+1},1}
\]
for $0\le a<k-1$, and
\[
\na^{k-1}P\na^{t}V_{l}\na^{2l+2-k-t}u\in L^{1}
\]
when $a=k-1$.

Hence, there holds
\[
\sum_{l\ge k/2}^{k-1}P\De^{l}(V_{l}\cdot\na u)=\sum_{a=0}^{k-1}\na^{k-1-a}\left(\sum_{l=k/2}^{k-1}\sum_{t=0}^{2l+1-k}K_{l,a,t}\na^{2l+2-k-t}u\right)
\]
for some functions
\[
K_{l,a,t}\in L^{\frac{2k}{2k-(2l+1-t)+a},1}
\]
for $k/2\le l\le k-1$, $0\le t\le2l+1-k$ and $0\le a\le k-1$.
Set $2l+2-k-t=i$ and $k-1-a=b$. We can further simplify the above
summation as
\begin{equation}
\sum_{l\ge k/2}^{k-1}P\De^{l}(V_{l}\cdot\na u)=\sum_{b=0}^{k-1}\na^{b}\left(\sum_{i=1}^{k}\tilde{K}_{b,i}\na^{i}u\right)\qquad \text{ in } B_{1/2}\label{eq: case 2}
\end{equation}
with coefficient functions $\tilde{K}_{b,i}\in L^{\frac{2k}{2k-b-i},1}(B_{1/2})$
for all $0\le b\le k-1,$ $1\le i\le k$.

Hence, we derive from \eqref{eq: case 1} and \eqref{eq: case 2} that
\[
\begin{aligned}\sum_{l=0}^{k-1}P\De^{l}(V_{l}\cdot\na u) & =-\na\De^{k-1}P\cdot\na u+K_{0}\cdot\na u\\
 & \quad+\sum_{a=0}^{k-1}\na^{a}\left(\sum_{i=1}^{k-a}K_{i,a}\na^{i}u\right)+\sum_{b=0}^{k-1}\na^{b}\left(\sum_{i=1}^{k}\tilde{K}_{b,i}\na^{i}u\right),
\end{aligned}
\]
where
\begin{eqnarray*}
K_{i,a}\in L^{2,1}(B_{1/2}) & \text{and} & \tilde{K}_{b,i}\in L^{\frac{2k}{2k-b-i},1}(B_{1/2}).
\end{eqnarray*}
Note that $K_{0}\in W^{2-k,\frac{2k}{k+1},1}\subset W^{1-k,2,1}$.
We can deal with $K_{0}\cdot\na u$ as in \textbf{Case I}. As a result, we
find that
\[
\sum_{l=0}^{k-1}P\De^{l}(V_{l}\cdot\na u)=-\na\De^{k-1}P\cdot\na u+\sum_{a=0}^{k-1}\na^{a}\left(\sum_{i=1}^{k-a}K_{i,a}\na^{i}u\right)+\sum_{b=0}^{k-1}\na^{b}\left(\sum_{i=1}^{k}\tilde{K}_{b,i}\na^{i}u\right)
\]
for some $K_{i,a}\in L^{2,1}$ with $1\le i\le k-a\le k$ and $\tilde{K}_{b,i}\in L^{\frac{2k}{2k-b-i},1}$
for all $0\le b\le k-1$ and $1\le i\le k$.

This summation can be further simplified. Indeed, note that $i\le k-a$
implies  $2\ge\frac{2k}{2k-i-a}$, from which it follows  $K_{i,a}\in L^{2,1}(B)\subset L^{\frac{2k}{2k-a-i},1}(B)$
for all $1\le i\le k-a\le k$. Thus, we can simplify the above summation
as
\begin{equation}
\sum_{l=0}^{k-1}P\De^{l}(V_{l}\cdot\na u)=-\na\De^{k-1}P\cdot\na u+\sum_{a=0}^{k-1}\na^{a}\left(\sum_{i=1}^{k}K_{i,a}\na^{i}u\right)\qquad \text{ in } B_{1/2},\label{eq: furthersimplified coefficients}
\end{equation}
under the assumption that
\[
K_{i,a}\in L^{\frac{2k}{2k-a-i},1}(B_{1/2})\qquad\text{for all }0\le a\le k-1,1\le i\le k.
\]
This regularity assumption turns out to be the critical regularity we shall need in the sequel.

Similarly, for terms involving $w_{l}$ ($0\le l\le k-2$), we may consider the following two cases
\begin{description}
\item[Case I] $2l+1\le k-1$;
\item[Case II] $2l+1\ge k$.
\end{description}

Doing the case study exactly as in the previous case, we find
\begin{equation*}
\sum_{l=0}^{k-2}P\De^{l}{\delta}(w_{l}\D u)=\sum_{a=0}^{k-1}\na^{a}\left(\sum_{i=1}^{k-a}K_{i,a}\na^{i}u\right)+\sum_{b=0}^{k-1}\na^{b}\left(\sum_{i=1}^{k}\tilde{K}_{b,i}\na^{i}u\right)
\end{equation*}
where $K_{i,a}\in L^{2,1}$ and $\tilde{K}_{b,i}\in L^{\frac{2k}{2k-b-i},1}$. This again can be simplified as
\begin{equation}
\sum_{l=0}^{k-2}P\De^{l}{\delta}(w_{l}\D u)=\sum_{a=0}^{k-1}\na^{a}\left(\sum_{i=1}^{k}\tilde{K}_{b,i}\na^{i}u\right) \qquad \text{ in } B_{1/2},\label{eq: rearrangement of w}
\end{equation}
under the assumption that
\[
\tilde{K}_{b,i}\in L^{\frac{2k}{2k-b-i},1}(B_{1/2}).
\]

Finally, combining \eqref{eq: First time rewrite equation}, \eqref{eq: furthersimplified coefficients}
and \eqref{eq: rearrangement of w} gives
\begin{equation}
\begin{aligned}\De^{k-1}(P\De u)= & J+\sum_{a=0}^{k-1}\na^{a}\left(\sum_{i=1}^{k}K_{a,i}\na^{i}u\right)  \qquad \text{ in } B_{1/2},\end{aligned}
\label{eq: rewritten of equation}
\end{equation}
where $J$ is defined as in \eqref{eq: def of J} and
\begin{eqnarray*}
K_{a,i}\in L^{\frac{2k}{2k-a-i},1}(B_{1/2}) &  & \text{for all }0\le a,i-1\le k-1.
\end{eqnarray*}

\subsection{Proof of Lemma  \ref{lem: decay of Long-Gastel}}


Let $\epsilon=\epsilon(k,m)$ be given as in Lemma \ref{lem: Gauge transform higher}. By the scaling invariance of the system \eqref{eq: Longue-Gastel system} (see e.g. \cite{deLongueville-Gastel-2019}), we can additionally assume that
$$\ta_0<\epsilon,$$ where $\ta_0$ is defined as in \eqref{notation: theta_0}.
Then all the functions appeared in the above subsection (including $P,\xi, K_{i,a}$ etc.) are well defined on the ball $B_{1/2}$, with norms bounded by $C\ep$. We first extend them from $B_{1/2}$  to the whole space $\R^{2k}$ with supports in $B_2$ such that their respective (Lorentz-Sobolev) norms in $\R^{2k}$ are no more than a bounded constant multiplying the corresponding norms in $B_{1/2}$.  For simplicity, we keep using the same notations for the extended  functions.


We start from the equation \eqref{eq: rewritten of equation}.
Let $I_{\alpha}$ be the fractional Riesz operator of order $\al\in (0,n)$. Set
\[
f=I_{2(k-1)}\left[J+\sum_{a=0}^{k-1}\na^{a}\left(\sum_{i=1}^{k}K_{a,i}\na^{i}u\right)\right] \qquad \text{ in }\R^{2k}
\]
and
\[
g=P\De u-f \qquad \text{ in }\R^{2k}.
\]

We first estimate the $L^{k,\wq}(\R^{2k})$ norm of $f$.

By the definition \eqref{eq: First time rewrite equation} of $J$,
we have
\[
I_{2(k-1)}J=\text{a linear combination of }\na^{2k-i-j}I_{2(k-1)}(\na^{i}P\na^{j}u),
\]
where $1\le i,j\le k$. With the convention that $I_{0}$ means some
singular integral operator, we obtain
\[
|\na^{2k-i-j}I_{2(k-1)}(\na^{i}P\na^{j}u)|\lesssim I_{i+j-2}(|\na^{i}P||\na^{j}u|).
\]
Since $\na^{i}P\in L^{{2k}/{i},2}$ and $\na^{j}u\in L^{{2k}/{j},\wq}$, H\"older's inequality gives $|\na^{i}P||\na^{j}u|\in L^{{2k}/{(i+j)},2}$. By the Riesz potential theorem, Proposition \ref{prop:Riesz potential}, we find that  $I_{i+j-2}(|\na^{i}P||\na^{j}u|)\in L^{k,\wq}$. Moreover,
\[
\||\na^{2k-i-j}I_{2(k-1)}(\na^{i}P\na^{j}u)\|_{L^{k,\wq}(\R^{2k})}\lesssim\|\nabla P\|_{W^{k-1,2}}\sum_{j=1}^{k}\|\na^{j}u\|_{L^{2k/j,\wq}}\lesssim\ep\sum_{j=1}^{k}\|\na^{j}u\|_{L^{2k/j,\wq}}.
\]
Taking summation over $i,j$, we deduce
\begin{equation}
\||I_{2(k-1)}(J)\|_{L^{k,\wq}(\R^{2k})}\lesssim\ep\sum_{j=1}^{k}\|\na^{j}u\|_{L^{2k/j,\wq}}\lesssim\ep\sum_{j=1}^{k}\|\na^{j}u\|_{L^{2k/j,\wq}(B_{1})}.\label{eq: estimate of potential of J}
\end{equation}

For $0\le a\le k-1$ and $1\le i\le k$, we have $K_{a,i}\na^{i}u\in L^{\frac{2k}{2k-a}}.$
Applying the Riesz potential theorem again, we deduce
\[
\|I_{2(k-1)}(\na^{a}(K_{a,i}\na^{i}u))\|_{L^{k,\wq}}\lesssim\|I_{2(k-1)-a}(|K_{a,i}\na^{i}u|)\|_{L^{k,\wq}}\lesssim\|K_{a,i}\|_{L^{\frac{2k}{2k-a-i},1}}\|\na^{i}u\|_{L^{2k/i,\wq}}.
\]
Thus,
\[
\left\|I_{2(k-1)}\Big(\sum_{a=0}^{k-1}\na^{a}\big(\sum_{i=1}^{k}K_{a,i}\na^{i}u\big)\Big)\right\|_{L^{k,\wq}}\lesssim\ep\sum_{j=1}^{k}\|\na^{j}u\|_{L^{2k/j,\wq}(B_{1})}.
\]
Consequently, we derive
\[
\left\Vert f\right\Vert _{L^{k,\wq}(\R^{2k})}\lesssim\ep\sum_{j=1}^{k}\|\na^{j}u\|_{L^{2k/j,\wq}(B_{1})}.
\]
In general, we obtain by the same method that, for each $1\le i\le k-2$
\[
\left\Vert \na^{i}f\right\Vert _{L^{k,\wq}(\R^{2k})}\lesssim\ep\sum_{j=1}^{k}\|\na^{j}u\|_{L^{2k/j,\wq}(B_{1})}.
\]

For the function $g$, note that $\De^{k-1}g=0$ on $B_{1/2}$. Thus,
for any $0<r<1/2$, we may apply \cite[Lemma 6.2]{Gastel-Scheven-2009CAG} to deduce
\[
\begin{aligned}\|g\|_{L^{k,\wq}(B_{r})} & \le C_{k}\|g\|_{L^{\wq}(B_{1/2})}r^{2}\lesssim\|g\|_{L^{1}(B_{1})}r^{2}\\
 & \lesssim r^{2}(\|P\De u\|_{L^{1}(B_{1})}+\|f\|_{L^{1}(B_{1})})\\
 & \lesssim r^{2}(\|\De u\|_{L^{k,\wq}(B_{1})}+\|f\|_{L^{k,\wq}(B_{1})})\\
 & \lesssim r^{2}\sum_{j=1}^{k}\|\na^{j}u\|_{L^{2k/j,\wq}(B_{1})},
\end{aligned}
\]
and
\[
\sum_{i=1}^{k-1}\|\na^{i}g\|_{L^{2k/i,\wq}(B_{r})}\lesssim r\sum_{j=1}^{k}\|\na^{j}u\|_{L^{2k/j,\wq}(B_{1})}.
\]
Therefore, it follows
\begin{equation}
\begin{aligned}\|\De u\|_{L^{k,\wq}(B_{r})} & \lesssim\|P\De u\|_{L^{k,\wq}(B_{r})}\\
 & \lesssim\|g\|_{L^{k,\wq}(B_{r})}+\|f\|_{L^{k,\wq}(B_{r})}\\
 & \lesssim(r^{2}+\ep)\sum_{j=1}^{k}\|\na^{j}u\|_{L^{2k/j,\wq}(B_{1})}.
\end{aligned}
\label{eq: second order estimate}
\end{equation}

The estimate of $\|\na u\|_{L^{2k,\wq}}$ is a little bit tricky.
We use the idea of Struwe \cite[formula (56)]{Struwe-2008} as follows. Split $u$ by $u=u_{0}+u_{1}$
in $B_{r}$, such that $u_{0}$ is a harmonic function in $B_{r}$
and $u_{1}=0$ on $\pa B_{r}$. Then, there exists a constant $C>0$
depending only on $k$, such that
\[
\|\na u_{0}\|_{L^{2k,\wq}(B_{\ga r})}\le C\ga\|\na u_{0}\|_{L^{2k,\wq}(B_{r})}
\]
for any $\ga\in(0,1)$. Since $\De u_{1}=\De u$ in $B_{r}$ with
zero boundary value, the $L^{p}$-theory (see e.g. \cite[Lemma 2.1]{Gastel-Scheven-2009CAG}) implies that
\[
r^{-1}\|\na u_{1}\|_{L^{k,\wq}(B_{r})}+\|\na^{2}u_{1}\|_{L^{k,\wq}(B_{r})}\lesssim\|\De u\|_{L^{k,\wq}(B_{r})}.
\]
So, using the Sobolev embedding $\|\na u_{1}\|_{L^{2k,\wq}(B_{r})}\lesssim r^{-1}\|\na u_{1}\|_{L^{k,\wq}(B_{r})}+\|\na^{2}u_{1}\|_{L^{k,\wq}(B_{r})}$,
we obtain
\[
\|\na u_{1}\|_{L^{2k,\wq}(B_{r})}\lesssim\|\De u\|_{L^{k,\wq}(B_{r})}.
\]
Therefore, for any fixed $\ga>0$, we derive
\[
\begin{aligned}\|\na u\|_{L^{2k,\wq}(B_{\ga r})} & \lesssim\|\na u_{0}\|_{L^{2k,\wq}(B_{\ga r})}+\|\na u_{1}\|_{L^{2k,\wq}(B_{\ga r})}\\
 & \lesssim\ga\|\na u_{0}\|_{L^{2k,\wq}(B_{r})}+\|\na u_{1}\|_{L^{2k,\wq}(B_{r})}\\
 & \lesssim\ga\|\na u\|_{L^{2k,\wq}(B_{r})}+\|\na u_{1}\|_{L^{2k,\wq}(B_{r})}\\
 & \lesssim\ga\|\na u\|_{L^{2k,\wq}(B_{r})}+\|\De u\|_{L^{k,\wq}(B_{r})}.
\end{aligned}
\]
Hence it follows from (\ref{eq: second order estimate}) that
\begin{equation}
\|\na u\|_{L^{2k,\wq}(B_{\ga r})}\lesssim(\ga+r^{2}+\ep)\sum_{j=1}^{k}\|\na^{j}u\|_{L^{2k/j,\wq}(B_{1})}\label{eq: first order estimate}
\end{equation}

To obtain estimates for higher order derivatives of $u$, we derive
from $P\De u=f+g$ that
\[
P\De\na u=-\na P\De u+\na f+\na g.
\]
This gives
\[
\begin{aligned}\|\De\na u\|_{L^{2k/3,\wq}(B_{r})} & \lesssim\|\na P\De u\|_{L^{2k/3,\wq}(B_{r})}+\|\na f\|_{L^{2k/3,\wq}(B_{r})}+\|\na g\|_{L^{2k/3,\wq}(B_{r})}\\
 & \lesssim(r^{2}+\ep)\sum_{j=1}^{k}\|\na^{j}u\|_{L^{2k/j,\wq}(B_{1})},
\end{aligned}
\]
which implies, by combining the $L^{p}$-theory and (\ref{eq: second order estimate})
(\ref{eq: first order estimate}), that

\[
\left\Vert \na^{3}u\right\Vert _{L^{2k,\wq}(B_{\ga r})}\lesssim(\ga+r^{2}+\ep)\sum_{j=1}^{k}\|\na^{j}u\|_{L^{2k/j,\wq}(B_{1})}.
\]

Repeating this process, we finally obtain a constant $C>0$ depending
only on $k,m$, such that

\[
\sum_{j=1}^{k}\|\na^{j}u\|_{L^{2k/j,\wq}(B_{\ga r})}\le C(\ga+r^{2}+\ep)\sum_{j=1}^{k}\|\na^{j}u\|_{L^{2k/j,\wq}(B_{1})}.
\]
Thus, choosing $\ep,r,\ga$ sufficiently small such that $C(\ga+r^{2}+\ep)\le1/2$,
and then letting $\tau=\ga r$, we obtain
\[
\sum_{j=1}^{k}\|\na^{j}u\|_{L^{2k/j,\wq}(B_{\tau})}\le\frac{1}{2}\sum_{j=1}^{k}\|\na^{j}u\|_{L^{2k/j,\wq}(B_{1})}.
\]
Finally, using a standard scaling argument, we obtain (\ref{eq: decay estimates of higher order system}).
The proof of Lemma \ref{lem: decay of Long-Gastel} is complete.

\section{A remaining open problem}\label{sec:obstacle}

Recall that Rivi\`ere and Struwe \cite{Riviere-Struve-2008} not only established their regularity result in the critical dimension $n=2$, but also a partial regularity theory in supercritical dimensions $n\ge 3$, which then generalized the well-known partial regularity theory of stationary harmonic mappings established by Evans \cite{Evans-1991} and Bethuel \cite{Bethuel-1993}. Thus, a natural question is to extend the  partial regularity theory of Rivi\`ere and Struwe \cite{Riviere-Struve-2008} to higher order elliptic systems in supercritical dimensions.  More precisely, the problem can be formulated as follows.
\begin{problem}[Supercritical case]\label{prob:supercrticial case} Consider the system \eqref{eq: Longue-Gastel system} in $B^n$, $n>2k$.
Does Theorem \ref{thm:general case} continue to hold in terms of partial regularity?
\end{problem}

In this section, we briefly discuss the obstacle to answer Problem \ref{prob:supercrticial case}. For simplicity, we shall only discuss the fourth order case and there is no essential difference regarding the higher order cases. Carefully checking the arguments in Section \ref{sec: Proof of main results}, it is plausible that our method is  flexible enough for the supercritical case, provided that we replace the Lorentz spaces used in the proof with the so called Lorentz-Morrey spaces.

Let $\Om\subset\R^{n}$ be an open set with smooth boundary. Let $1\le p<\wq$
and $0\le s<n$. The \emph{Morrey space} $M^{p,s}(\Om)$ consists
of functions $f\in L^{p}(\Om)$ such that
\[
\|f\|_{M^{p,s}(\Om)}\equiv\left(\sup_{x\in\Om,r>0}r^{-s}\int_{B_{r}(x)\cap\Om}|f|^{p}\right)^{1/p}<\wq.
\]
The \emph{$k$-th order Morrey space} $M_{k}^{p,n-kp}(\Om)$ consists of $f\in W^{k,p}(\Om)$
such that $\na^{l}u\in M^{p,n-lp}(\Om)$.

For $1<p<\infty$, we define the \emph{weak Morrey space} $M^{p,s}_*(\Omega)$ as the space of functions $f\in L^{p,\infty}(\Omega)$ such that
$$\|f\|_{M^{p,s}_*(\Om)}
\equiv\left(\sup_{x\in\Om,r>0}r^{-s}\|f\|^p_{L^{p,\wq}(B_r(x)\cap \Omega)}\right)^{1/p}<\wq.$$

Let $1\le p<\wq$ and $1\le q\le\wq$, $0<s<n$.
The \emph{Lorentz-Morrey space} $LM^{p,q,s}(\Om)$ consists of functions $u\in L^{p,q}(\Om)$ such that
\[
\|u\|_{LM^{p,q,s}(\Om)}^{p}\equiv\sup_{x\in\Om,r>0}r^{-s}\|u\|_{L^{p,q}(B_{r}(x)\cap\Om)}^{p}.
\]
It is easy to see that $LM^{p,p,s}=M^{p,s}$, $LM^{p,\infty,s}=M_{*}^{p,s}$ and since $L^{p,q_{1}}\subset L^{p,q_{2}}$
if $q_{1}<q_{2}$, we have
\[
LM^{p,q_{1},s}\subset LM^{p,q_{2},s}.
\]

As we have already seen in the proof of Theorem~\ref{thm:main thm}, an essential ingredient we used is the Lorentz-Sobolev embedding 
\begin{equation}\label{eq: Tartar}
W^{1,2}(\R^4) \hookrightarrow L^{4,2}(\R^4).
\end{equation}  In supercritical dimensions,  the following embedding (if held) seems to be a suitable replacement of \eqref{eq: Tartar}:
\begin{equation}\label{eq: embedding-1}
M_{1}^{2,n-4}\cap M^{4,n-4}(\R^n)\hookrightarrow LM^{4,2,n-4}(\R^n).
\end{equation}
When $n=4$, it reduces to  \eqref{eq: Tartar}.


A possible approach to the above embedding is as follows. Suppose $f\in M_{1}^{2,n-4}\cap M^{4,n-4}(\R^{n})$ has compact support and let $\Gamma$ be the fundamental solution of $-\Delta$. Then, we have the following pointwise estimate
\begin{equation*}\label{eq:estimate of f by Riesz potential}
f=\Ga\ast\De f=\na\Ga\ast\na f\approx I_{1}(|\na f|).
\end{equation*}
If we were able to show the boundedness of
\begin{equation}\label{eq:Needed Riesz bound}
I_{1}\colon LM^{2,2,n-4}(\R^n)\to LM^{4,2,n-4}(\R^{n})\footnote{Note that this is true when $n=4$.},
\end{equation}
then it would follow that
\[
\begin{aligned}
\|f\|_{LM^{4,2,n-4}(\R^n)}&\lesssim \|I_1(\nabla f)\|_{LM^{4,2,n-4}(\R^n)}\lesssim \|\nabla f\|_{LM^{2,2,n-4}(\R^n)}<\infty
\end{aligned}
\]
and so \eqref{eq: embedding-1} follows.

Inspired by the above analysis and taking into consideration of the fractional Riesz operator between Lorentz/Morrey spaces, we would like to put up the following conjecture, which might have potential applications elsewhere as well.
\medskip

\textbf{Conjecture}: Let $0\leq \alpha<n$, $1<p_0<\frac{n}{\alpha}$, $0<\lambda<n-\alpha p_0$. If
\begin{equation*}\label{eq:sharp relation LM}
	\frac{1}{p_0}-\frac{1}{p_1}=\frac{\alpha}{n-\lambda}
\end{equation*}
and
\begin{equation}\label{eq:relation LM 2}
	q_1\geq q_0
\end{equation}
hold, then
$$I_\alpha \colon LM^{p_0,q_0,\lambda}(\R^n)\to LM^{p_1,q_1,\lambda}(\R^n)$$
is bounded.

\begin{remark}\label{rmk:on the claim}
	i). If $\lambda=0$, then the above claim covers the sharp boundedness of the fractional Riesz operator between Lorentz spaces.
	
	ii). If $p_0=q_0$ and $p_1=q_1$, then the above claim covers the sharp boundedness of the fractional Riesz operator between Morrey spaces.
	
	
	iii). If we replace \eqref{eq:relation LM 2} by $\frac{p_0}{q_0}\geq \frac{p_1}{q_1}$, then \textbf{Conjecture} holds by  \cite[Theorem 5.1]{Ho-2014}.
\end{remark}

Note that if the \textbf{Conjecture} were true, then applying it with $p_0=2$, $p_1=4$, $\alpha=1$, $\lambda=n-4$ and $q_0=q_1=2$, would give the boundedness of $I_1\colon LM^{2,2,n-4}\to LM^{4,2,n-4}$, which is our desired \eqref{eq:Needed Riesz bound}.

\appendix

\section{Regularity assumptions  in the biharmonic case}\label{subsec:regularity assumptions}
In this appendix, we show that our regularity assumptions \eqref{eq: coefficient G n=4}, \eqref{eq: coefficients D-E n=4} and \eqref{eq: coefficients Omega n=4}  are automatically  satisfied in the  case when $u\colon B_{10}^{4}\to N\subset\R^{m}$ is a stationary biharmonic mapping, where $N\subset \R^m$ is a closed Riemannian manifold. 

Following the computation of Struwe (see in particular \cite[Equations (12), (14), (16)]{Struwe-2008}), we know that $u$ satisfies the linear system \eqref{eq: Lamm-Riviere system n=4} with
\[
D=\big(D_{\al}^{i,j}\big)=\big(3(w^{j}\pa_{\al}w^{i}-w^{i}\pa_{\al}w^{j})\big),
\]
\[
E=\big(E_{\alpha,\beta}^{i,j}\big)=\big(\left(4\pa_{\al\be}(w^{i}w^{j})-\de_{\al\be}\De(w^{i}w^{j})\right)\big),
\]
and $F=\De\Om+G$ with
\begin{eqnarray*}
\Om=\big(\Omega_{ij} \big)=\big((w^{i}\na w^{j}-w^{j}\na w^{i})\big)
\end{eqnarray*}
and $G$ involves sums of terms like $\nabla^2 w^i\nabla w^j$, where $w^{i}=(\nu_{1}^{i}(u),\cdots,\nu_{k}^{i}(u))$ with $\{\nu_{i}\}_{i\le k}$ being a smooth local normal vector frame of the target manifold.

(1) $D$ is a linear combination of $w^{i}\pa_{l}\nu^{j}(u)\pa_{\al}u^{l}$.
This implies that $|D|\le C|\na u|$ and $|\na D|\le C(|\na^{2}u|+|\na u|^{2})$.
Thus, we get $D\in W^{1,2}$.

(2) From the expression of $E$, we know $|E|\le C(|\na^{2}u|+|\na u|^{2})$.
Thus, $E\in L^{2}$.

(3) From the expression of  $G$, we know
$|G|\le C(|\na^{2}u||\na u|+|\na u|^{3})$. Since $\nabla u\in W^{1,2}$, by the Lorentz-Sobolev embedding, $\nabla u\in L^{4,2}$.
Applying Proposition \ref{prop: Lorentz-Holder inequality} (with $p_1=q_1=2$ and $p_2=4,q_2=2$), we infer that
\[
G\in L^{\frac{4}{3},1}.
\]

(4) From the expression of $\Om$, we know $|\Om|\le C|\na u|$ and $|\na\Om|\le C(|\na^{2}u|+|\na u|^{2})$.
So $\Om$ has the same regularity assumption as $D$: $\Om\in W^{1,2}$.
In summary,  regularity assumptions \eqref{eq: coefficient G n=4}, \eqref{eq: coefficients D-E n=4} and \eqref{eq: coefficients Omega n=4} are satisfied in this case.

\end{document}